\documentclass[preprint,12pt]{elsarticle}
\biboptions{sort&compress}

\usepackage{amsmath,amssymb,amsthm,mathtools}
\usepackage{enumitem}
\usepackage[
  a4paper,
  left=2.4cm,
  right=2.4cm,
  top=2.3cm,
  bottom=2.3cm
]{geometry}
\usepackage[colorlinks=true,linkcolor=blue,citecolor=blue,urlcolor=blue]{hyperref}

\makeatletter
\renewenvironment{keyword}{%
  \def\sep{\unskip, }%
  \def\MSC{\@ifnextchar[{\@MSC}{\@MSC[2000]}}%
  \def\@MSC[##1]{\par\medskip\noindent
    \textbf{AMS Subject Classification (##1): }\ignorespaces}%
  \global\setbox\keybox=\vbox\bgroup\hsize=\textwidth
  \normalsize\normalfont\parskip\z@
  \noindent\textbf{Key Words: }\raggedright\ignorespaces
}{\par\egroup}
\makeatother

\makeatletter
\def\ps@pprintTitle{%
  \let\@oddhead\@empty
  \let\@evenhead\@empty
  \let\@oddfoot\@empty
  \let\@evenfoot\@empty
}
\makeatother

\allowdisplaybreaks[2]
\setlist{nosep,leftmargin=2em}

\newtheorem{theorem}{Theorem}[section]
\newtheorem{lemma}[theorem]{Lemma}
\newtheorem{corollary}[theorem]{Corollary}

\newtheorem{problem}{Problem}
\newtheorem{definition}{Definition}

\newcommand{\spex}{\operatorname{spex}_{\mathcal{OP}}}

\newcommand{\HOP}{H_{\mathcal{OP}}}
\newcommand{\HOPtwo}{\widetilde{H}_{\mathcal{OP}}}

\begin{document}

\begin{frontmatter}

\title{\vspace*{-2cm}\LARGE Extremal spectral result of outerplanar graphs without $P_{3\cdot l}$ \vspace{1.2cm}}

\author{{\large Fulong Ye$^{a,b,c}$, Yuxiang Liu$^{a,c}$, Ligong Wang$^{a,c,\dagger}$}\\[0.8em]
		{\small $^a$School of Mathematics and Statistics, Northwestern
			Polytechnical University,}\\ {\small  Xi'an, Shaanxi 710129,
			P.R. China.}\\
        {\small $^b$School of Mathematics and Statistics, Qinghai Minzu University,}\\ {\small  Xining, Qinghai 810007,
			P.R. China.}\\
		{\small $^c$Xi'an-Budapest Joint Research Center for Combinatorics, Northwestern
			Polytechnical University,}\\
		{\small Xi'an, Shaanxi 710129,
			P.R. China. }\\
		{\small E-mail: yefl\_qhmu@163.com, yxliumath@163.com, lgwangmath@163.com} }

\nonumnote{\textsuperscript{$*$}Supported by the National Natural Science
Foundation of China (No.~12271439).}
\nonumnote{\textsuperscript{$\dagger$}Corresponding author.}

\begin{abstract}
A graph $G$ is $F$-free if it does not contain $F$ as a subgraph. Let $\spex(n,F)$ be the maximum spectral radius over all $n$-vertex $F$-free outerplanar graphs. For integers $t\geq1$ and $l\geq2$, let $P_{t\cdot l}$ be the starlike tree with $t$ branches of length $l-1$. For sufficiently large $n$, Yin, Li, and Meng  [arXiv:2504.04364v1] characterized the unique extremal graph for $\spex(n,P_{t\cdot l})$ when $t=1$, $t=2$, or $t\geq4$. They left the case $t=3$ open and proposed a natural candidate for the extremal graph. We show that this candidate is not extremal and determine the unique extremal graph for $\spex(n,P_{3\cdot l})$. For every $l\geq3$ and all sufficiently large $n$, this unique extremal graph is
$
K_1\vee\bigl(2P_{2l-3}\cup qP_{l-2}\cup P_r\bigr),
$
where $q$ and $r$ are integers satisfying
$
n=2(2l-3)+q(l-2)+r+1,$ $q\geq0,$ $0\leq r<l-2.
$
\end{abstract}


\end{frontmatter}

\section{Introduction}
Throughout this paper, all graphs are finite, simple, and undirected. For a graph $G$, let $V(G)$ and $E(G)$ denote its vertex set and edge set, respectively. The order and size of $G$ are $|V(G)|$ and $e(G)=|E(G)|$, respectively. For a vertex $v\in V(G)$, let $N_G(v)$ denote the neighborhood of $v$, the degree of $v$ is  defined as $d_G(v)=|N_G(v)|$. Let $A(G)$ denote the adjacency matrix of $G$. The largest eigenvalue of $A(G)$, denoted by $\rho(G)$, is called the spectral radius of $G$.
For two vertex-disjoint graphs $G$ and $H$, let $G\cup H$ denote their union. The join $G\vee H$ is obtained from $G\cup H$ by adding all edges between $V(G)$ and $V(H)$.
For a positive integer $k$, let $kG$ denote the union of $k$ disjoint copies of $G$.
As usual, $P_n$, $C_n$ and $K_n$ denote the path, cycle and complete graph of order $n$, respectively.
A linear forest $F(l_1,l_2,\ldots,l_t)$ is the union of $t$
vertex-disjoint paths $P_{l_1},P_{l_2},\ldots,P_{l_t}$, where
$F(l_1,l_2,\ldots,l_t)=\bigcup_{i=1}^tP_{l_i}$ and
$l_1\geq l_2\geq\cdots\geq l_t\geq1$.
For a vertex subset $S\subseteq V(G)$, let $G[S]$ denote the subgraph of $G$ induced by $S$. We write
$G-S=G[V(G)\setminus S].$
When $S=\{v\}$, we simply write $G-v$ instead of $G-\{v\}$.
For $u\in V(G)$ and $v\notin V(G)$, let $G+uv$ denote the graph obtained
from $G$ by adding the new vertex $v$ and the edge $uv$.
A graph is planar if it can be drawn in the plane without edge crossings.
A planar graph is outerplanar if it has a plane drawing in which all vertices
lie on the boundary of the outer face.
We refer to Bondy and Murty~\cite{BondyMurty} for standard notation and terminology not defined here.

A graph is $F$-free if it does not contain $F$ as a subgraph. A basic problem in spectral graph theory is to determine the
maximum spectral radius in a given graph class and to characterize the graphs
attaining it. This problem is known as the Brualdi--Solheid
problem~\cite{BrualdiSolheid}. The same question has been studied for
$F$-free graphs and for many choices of $F$, including complete
graphs~\cite{BollobasNikiforov}, matchings~\cite{FengYuZhang}, paths and
cycles~\cite{NikiforovPaths,ZhaiLin}, friendship
graphs~\cite{CioabaFengTaitZhang}, wheels~\cite{ZhaoHuangLin}, star
forests~\cite{ChenLiuZhangStar}, and linear
forests~\cite{ChenLiuZhangLinear}. For surveys and further results concerning
graph spectra, see~\cite{ChenZhangSurvey,LiLiuFengSurvey,NikiforovSurvey}.

Spectral extremal questions have also been studied for planar and outerplanar
graphs. In 1978, Schwenk and Wilson~\cite{SchwenkWilson} asked what could be
said about the eigenvalues of a planar graph. Boots and
Royle~\cite{BootsRoyle}, and independently Cao and Vince~\cite{CaoVince},
later conjectured that $K_2\vee P_{n-2}$ has the maximum spectral radius among
all planar graphs of order $n\geq 9$. In 2017, Tait and
Tobin~\cite{TaitTobin} proved this conjecture for sufficiently large $n$.
In 1990, Cvetkovi\'{c} and Rowlinson~\cite{CvetkovicRowlinson} conjectured that
$K_1\vee P_{n-1}$ is the unique outerplanar graph of order $n$ with maximum
spectral radius. Tait and Tobin~\cite{TaitTobin} proved the conjecture for
sufficiently large $n$. In 2021, Lin and Ning~\cite{LinNing} proved that the conjecture holds for every $n\geq 2$ except $n=6$. We refer to
Felsner~\cite{Felsner} for basic facts about outerplanar graphs.

The discussion above naturally leads to the following problem concerning forbidden subgraphs.
Let
$\operatorname{spex}_{\mathcal{P}}(n,F)$ (resp.\
$\operatorname{spex}_{\mathcal{OP}}(n,F)$) denote the maximum spectral radius
of any $F$-free planar graph (resp.\ outerplanar graph) on $n$ vertices, and
let $\operatorname{SPEX}_{\mathcal{P}}(n,F)$ (resp.\
$\operatorname{SPEX}_{\mathcal{OP}}(n,F)$) denote the corresponding set of
extremal graphs.

In 2024, Fang, Lin, and Shi~\cite{FangLinShi} determined
$\operatorname{spex}_{\mathcal{P}}(n,tC_l)$ and
$\operatorname{spex}_{\mathcal{P}}(n,t\mathcal{C})$, and characterized the
unique extremal graphs when $1\leq t\leq2$, $l\geq3$, and $n$ is
sufficiently large, where $tC_l$ is the vertex-disjoint union of $t$
cycles of order $l$, and $t\mathcal{C}$ denotes the family of vertex-disjoint
unions of $t$ cycles with unrestricted lengths.
In the same year, Zhai and Liu~\cite{ZhaiLiu} determined
$\operatorname{SPEX}_{\mathcal{P}}(n,\mathcal F_k)$ for sufficiently large
$n$, where $\mathcal F_k$ is the family of graphs consisting of $k$
edge-disjoint cycles. In 2025, Wang, Huang, and Lin~\cite{WangHuangLin} determined
$\operatorname{SPEX}_{\mathcal{P}}(n,W_k)$,
$\operatorname{SPEX}_{\mathcal{P}}(n,F_k)$, and
$\operatorname{SPEX}_{\mathcal{P}}(n,(k+1)K_2)$ for sufficiently large $n$,
where $W_k=K_1\vee C_{k-1}$ is the wheel graph of order $k$ and
$F_k=K_1\vee kK_2$ is the friendship graph of order $2k+1$.

We use the following notation introduced by Yin and Li~\cite{YinLi}.
\begin{definition}\label{def:HOP}
For positive integers $n,n_1,n_2$ with $n-1\geq n_1\geq n_2\geq1$, define
\[
\HOP(n_1,n_2)=
\begin{cases}
P_{n_1}\displaystyle\cup\frac{n-1-n_1}{n_2}P_{n_2},
&\text{if }n_2\mid(n-1-n_1),\\[1ex]
P_{n_1}\displaystyle\cup
\left\lfloor\frac{n-1-n_1}{n_2}\right\rfloor P_{n_2}
\displaystyle\cup
P_{n-1-n_1-\left\lfloor\frac{n-1-n_1}{n_2}\right\rfloor n_2},
&\text{otherwise}.
\end{cases}
\]
\end{definition}

We also use the following notation for a linear forest containing two copies of its longest path.
\begin{definition}\label{def:HOPtwo}
For positive integers $n,n_1,n_2$ with $n-1\geq2n_1$ and $n_1\geq n_2\geq1$,
define
\[
\HOPtwo(n_1,n_2)=
\begin{cases}
2P_{n_1}\displaystyle\cup\frac{n-1-2n_1}{n_2}P_{n_2},
&\text{if }n_2\mid(n-1-2n_1),\\[1ex]
2P_{n_1}\displaystyle\cup
\left\lfloor\frac{n-1-2n_1}{n_2}\right\rfloor P_{n_2}
\displaystyle\cup
P_{n-1-2n_1-\left\lfloor\frac{n-1-2n_1}{n_2}\right\rfloor n_2},
&\text{otherwise}.
\end{cases}
\]
\end{definition}

In 2026, Yin and Li~\cite{YinLi} established a key structural theorem for a
broad class of forbidden outerplanar graphs. This result appears as
Theorem~1 in their paper and is stated as Lemma~\ref{lem:structure} below.
Under the conditions of this theorem, every connected extremal graph $G$ for
$\operatorname{spex}_{\mathcal{OP}}(n,F)$ has a vertex $u$ adjacent to all
other vertices. Moreover, $G-u$ is a linear forest.
Using this structural theorem, Yin and Li~\cite{YinLi} determined
$\operatorname{spex}_{\mathcal{OP}}(n,B_{t,l})$ and
$\operatorname{spex}_{\mathcal{OP}}(n,(t+1)K_2)$, and characterized the corresponding unique extremal graphs under the parameter ranges specified in their theorem, for all sufficiently large $n$,
where $B_{t,l}$ is obtained by sharing a common
vertex among $t$ edge-disjoint $l$-cycles.

In 2025, Yin, Li, and Meng~\cite{YinLiMeng} studied forbidden paths
and starlike trees. For integers $t\geq 1$ and $l\geq 2$, let
$tP_l$ be the linear forest consisting of $t$ copies of $P_l$, and let
$P_{t\cdot l}$ be the starlike tree with $t$ branches of length $l-1$. For
sufficiently large $n$, they determined the extremal graphs in
$\operatorname{SPEX}_{\mathcal{P}}(n,tP_l)$ when $t=1$ and $l\geq 6$, when
$t=2$ and $l\geq 4$, or when $t\geq 3$ and $l\geq 3$. They also characterized
$\operatorname{SPEX}_{\mathcal{OP}}(n,tP_l)$ when $t=1$ and $l\geq 4$, or
when $t\geq 2$ and $l\geq 3$. Their result for starlike trees is stated below.

\begin{theorem}[{\cite{YinLiMeng}}]\label{thm:YinLiMeng}
Let $n_0=\max\left\{1.27\times10^7, 6.5025\times2^{(t-1)(l-1)+l}\right\}.$
\begin{enumerate}[label=\textup{(\roman*)}]
\item If $t=1$, $l\geq4$, and
$n\geq\max\left\{n_0,
\left(5.08\left\lfloor\frac{l-2}{2}\right\rfloor\right)^2+1\right\}$, then
$K_1\vee\HOP\left(
\left\lceil\frac{l-2}{2}\right\rceil,
\left\lfloor\frac{l-2}{2}\right\rfloor\right)$
is the unique extremal graph for $\spex(n,P_{t\cdot l})$.
\item If $t=2$, $l\geq3$, and
$n\geq\max\left\{n_0,
\left(5.08\left\lfloor\frac{2l-3}{2}\right\rfloor\right)^2+1\right\}$, then
$K_1\vee\HOP\left(
\left\lceil\frac{2l-3}{2}\right\rceil,
\left\lfloor\frac{2l-3}{2}\right\rfloor\right)$
is the unique extremal graph for $\spex(n,P_{t\cdot l})$.
\item If $t\geq4$, $l\geq3$, and $n\geq n_0$, then $K_1\vee\HOP(tl-t-1,l-2)$ is the unique extremal graph for
$\spex(n,P_{t\cdot l})$.
\end{enumerate}
\end{theorem}

They left the case \(t=3\) open and asked whether the following natural candidate is the unique extremal graph.
\begin{problem}[{\cite{YinLiMeng}}]\label{prob:YLM}
Let $t=3$, $l\geq3$, and let $n$ be sufficiently large. Is $K_1\vee\HOP(2l-2,l-2)$ the unique extremal graph for
$\spex(n,P_{3\cdot l})$?
\end{problem}

We solve this problem by showing that the proposed graph is not extremal.
The unique extremal graph is given in Theorem~\ref{thm:main}.
\begin{theorem}\label{thm:main}
If $l\geq3$ and
$
n\ge \max\left\{
1.27\times 10^7,\,
6.5025\times 2^{3l-2},\,
(\sqrt{62+3l}+8)^2
\right\},
$
then $K_1\vee\HOPtwo(2l-3,l-2)$ is the unique extremal graph for
$\spex(n,P_{3\cdot l})$.
\end{theorem}

The remainder of the paper is organized as follows.
Section~\ref{sec:prelim} establishes the preliminary spectral and structural results. Section~\ref{sec:proof} then proves Theorem~\ref{thm:main}.

\section{Preliminaries}\label{sec:prelim}
For a graph $G$, let $\boldsymbol{x}=(x_1,x_2,\ldots,x_n)^{\mathrm T}$ be an eigenvector of
$A(G)$ corresponding to $\rho(G)$. For each vertex $v\in V(G)$, the
eigenvalue equation is
\[
\rho(G)x_v=\sum_{u\in N_G(v)}x_u.
\]
By the Rayleigh principle,
\[
\rho(G)
=\max_{\boldsymbol{x}\in\mathbb{R}^n\setminus\{\boldsymbol{0}\}}
\frac{\boldsymbol{x}^{\mathrm T}A(G)\boldsymbol{x}}
     {\boldsymbol{x}^{\mathrm T}\boldsymbol{x}}
=\max_{\boldsymbol{x}\in\mathbb{R}^n\setminus\{\boldsymbol{0}\}}
\frac{2\displaystyle\sum_{uv\in E(G)}x_ux_v}
     {\boldsymbol{x}^{\mathrm T}\boldsymbol{x}}.
\]
If $G$ is connected, the Perron--Frobenius theorem guarantees a
positive eigenvector
$\boldsymbol{x}=(x_1,x_2,\ldots,x_n)^{\mathrm T}$
corresponding to $\rho(G)$.
Consequently, if $G$ and $G'$ have the same vertex set and $\boldsymbol{x}$
is a positive eigenvector corresponding to $\rho(G)$, then
\begin{equation}\label{eq:edgecomparison}
 \rho(G')-\rho(G)\ge
 \frac{2}{\boldsymbol{x}^{\mathrm T}\boldsymbol{x}}
 \left(
 \sum_{vw\in E(G')\setminus E(G)}x_vx_w-
 \sum_{vw\in E(G)\setminus E(G')}x_vx_w
 \right).
\end{equation}

\begin{lemma}\label{lem:standard}
The following statements hold.
\begin{enumerate}[label=\textup{(\roman*)}]
  \item (\cite{Felsner}) Every outerplanar graph $G$ of order $n\ge2$
 satisfies $e(G)\le2n-3$.
  \item (\cite{Felsner}) If $u$ and $v$ are distinct vertices of an outerplanar graph $G$, then
$
\lvert N_G(u)\cap N_G(v)\rvert\le 2.
$
 \item (\cite{LiFeng}) If $G$ is connected and $H$ is a
 proper spanning subgraph of $G$, then $\rho(H)<\rho(G)$.
\end{enumerate}
\end{lemma}

\begin{lemma}[{\cite{YinLi}}]\label{lem:structure}
Let $F$ be an outerplanar graph contained in $K_1\vee P_{n-1}$ but not
contained in $K_1\vee((t-1)K_2\cup(n-2t+1)K_1)$, where
$1\leq t\leq \frac{n-1}{2}$ and
$n\geq\max\{1.27\times10^7,[(64+|V(F)|)^{\frac{1}{2}}+8]^2\}$.
Suppose that $G$ is a connected extremal graph for
$\operatorname{spex}_{\mathcal{OP}}(n,F)$
and $\boldsymbol{x}$ is a positive eigenvector corresponding to $\rho(G)$
with $\max_{v\in V(G)}x_v=1$. Then there exists a vertex
$u\in V(G)$ such that $N_G(u)=V(G)\setminus\{u\}$ and $x_u=1$. In particular,
\begin{enumerate}[label=(\roman*)]
    \item $G$ contains a copy of $K_{1,n-1}$.
    \item The subgraph $G[N_G(u)]$ is a linear forest.
\end{enumerate}
\end{lemma}

\begin{lemma}[{\cite{YinLi}}]\label{lem:coordinates}
Let $n\ge1.27\times10^7$ and $G=K_1\vee H$, where $H$ is a linear forest of order $n-1$.
Let $u$ be the vertex of degree $n-1$, and let
$\boldsymbol{x}$ be a positive eigenvector corresponding to $\rho(G)$ with
$\max_{v\in V(G)}x_v=1$.
If $\rho=\rho(G)$, then, for every $v\in V(H)$,
\[
 \frac1\rho\le x_v\le\frac1\rho+\frac{2.04}{\rho^2}.
\]
\end{lemma}

\begin{lemma}[{\cite[Outerplanar path transformation]{YinLi}}]\label{lem:transfer}
Let $H=P_a\cup P_b\cup H_0$, where $a\ge b\ge1$ and $H_0$ is a linear forest,
and let $H'=P_{a+1}\cup P_{b-1}\cup H_0$, where $P_0$ is omitted when $b=1$. If
$n\ge\max\{1.27\times10^7,6.5025\times 2^{b+2}\}$, then
$\rho(K_1\vee H')>\rho(K_1\vee H)$.
\end{lemma}

If $G$ has components $G_1,\ldots,G_k$, then $A(G)$ is the  block diagonal and
\[\rho(G)=\max_{1\leq i\leq k}\rho(G_i).\]

\begin{lemma}\label{lem:replacement}
Let $F$ be a connected graph of order $k\ge2$, and let $G$ be an $F$-free
graph. Let $v\notin V(G)$, $u\in V(G)$, and $G'=G+uv$. If
$d_G(u)\ge k-1$, then $G'$ is $F$-free.
\end{lemma}

\begin{proof}
Suppose that $G'$ contains a copy of $F$. Since $G$ is $F$-free,
this copy must contain $v$ and the edge $uv$. Moreover,
$|V(F)\setminus\{u,v\}|=k-2$. Since $d_G(u)\geq k-1$, there exists
$w\in N_G(u)\setminus V(F)$. Replacing $v$ and $uv$ in $F$ with $w$
and $uw$, respectively, gives a copy of $F$ in $G$, a contradiction.
\end{proof}
\begin{lemma}\label{lem:connected}
Let $l\ge3$ and $n\ge\max\left\{1.27\times10^7,\,
(\sqrt{62+3l}+8)^2\right\}.$
Then every extremal graph for $\spex(n,P_{3\cdot l})$ is connected.
\end{lemma}

\begin{proof}
Suppose to the contrary that $G$ is disconnected. Let $G_1$ be a component
of $G$ such that $\rho(G_1)=\rho(G)$, and write $\rho=\rho(G)$.
The star $K_{1,n-1}$ is $P_{3\cdot l}$-free. Hence
$\rho\ge\sqrt{n-1}\ge\sqrt{12\,699\,999}>3563>36$.

Take a positive eigenvector of $G_1$ corresponding to $\rho(G_1)$ and extend
it by zero to $V(G)$. Denote the resulting vector by $\boldsymbol{x}$ and choose
it so that $\max_{v\in V(G)}x_v=1$. Let $u\in V(G_1)$ satisfy $x_u=1$.
Summing the eigenequations over all vertices gives
\[
 \rho\sum_vx_v=\sum_vd_G(v)x_v\le2e(G).
\]
Applying the eigenvalue equation twice at $u$ and then using
Lemma~\ref{lem:standard}(ii), we obtain
\begin{align*}
 \rho^2
 &=\sum_{y\in N_G(u)}\sum_{z\in N_G(y)}x_z\notag\\
 &=d_G(u)+\sum_{z\ne u}|N_G(u)\cap N_G(z)|x_z\notag\\
 &\le d_G(u)+2\sum_zx_z
 \le d_G(u)+\frac{4e(G)}{\rho}.
\end{align*}
By Lemma~\ref{lem:standard}(i) and $\rho\ge\sqrt{n-1}$,
\begin{align*}
 d_G(u)
 &\ge n-1-\frac{8n-12}{\sqrt{n-1}}\notag\\
 &=n-1-8\sqrt{n-1}+\frac4{\sqrt{n-1}}
 >n-16\sqrt n\ge3l-2.
\end{align*}
The last inequality follows from
$n\ge(\sqrt{62+3l}+8)^2$.

Now choose a component $G_i\ne G_1$. Delete all edges of $G_i$, select
$v\in V(G_i)$, and add the edge $uv$. Denote the resulting graph by $G'$.
Adding a pendant edge preserves outerplanarity.
Since $d_{G_1}(u)=d_G(u)\geq3l-2\geq |V(P_{3\cdot l})|-1$, Lemma~\ref{lem:replacement} shows that $G_1+uv\) is \(P_{3\cdot l}$-free.
Thus, $G'$ is a $P_{3\cdot l}$-free outerplanar graph.
Since $G_1$ together with the isolated vertex $v$ is a proper spanning
subgraph of the connected graph $G_1+uv$, Lemma~\ref{lem:standard}(iii) gives
\[
  \rho(G')\ge\rho(G_1+uv)>\rho(G_1\cup\{v\})=\rho(G_1)=\rho(G),
\]
contrary to the extremality of $G$. Therefore, $G$ is connected.
\end{proof}

For $l\ge3$, we have $P_{3\cdot l}\subseteq K_1\vee P_{n-1}$ and
$P_{3\cdot l}\nsubseteq K_1\vee(2K_2\cup(n-5)K_1)$. Therefore,
Lemmas~\ref{lem:connected} and~\ref{lem:structure} yield the following
result.

\begin{corollary}\label{cor:join}
If $l\geq3$ and
$
n\geq\max\left\{
1.27\times10^7,\,
6.5025\times2^{3l-2},\,
(\sqrt{62+3l}+8)^2
\right\},
$
then any $P_{3\cdot l}$-free outerplanar graph $G$ of order $n$ satisfying
$
\rho(G)=\spex(n,P_{3\cdot l}),
$
has the form
$
G=K_1\vee H,
$
where $H$ is a linear forest of order $n-1$.
\end{corollary}

\section{Proof of the main theorem}\label{sec:proof}

We first give an exact characterization of the $P_{3\cdot l}$-free graphs of
the form $K_1\vee H$, where $H$ is a linear forest.

\begin{lemma}\label{lem:criterion}
For $H=\bigcup_{i=1}^kP_{s_i}$ with $s_i\ge1$, the graph $K_1\vee H$ is $P_{3\cdot l}$-free if and only if
\begin{equation}\label{eq:capacity}
 \sum_{i=1}^k\left\lfloor\frac{s_i}{l-1}\right\rfloor\le2
\end{equation}
and
\begin{equation}\label{eq:long}
 \text{if }s_i\ge2l-1,\quad\text{then }s_j\le l-3
 \quad\text{for all distinct }i,j.
\end{equation}
\end{lemma}
\begin{proof}
Let $G=K_1\vee H$, and let $u$ be the vertex of $K_1$.

Suppose first that \eqref{eq:capacity} fails. Then $H$ contains three
vertex-disjoint subpaths $Q_1,Q_2,Q_3$, each of order $l-1$. Hence
$G[\{u\}\cup V(Q_1)\cup V(Q_2)\cup V(Q_3)]$ contains a copy of
$P_{3\cdot l}$ whose unique vertex of degree three is $u$.

Next, suppose that \eqref{eq:long} fails. Then there exist distinct
$i,j\in\{1,\ldots,k\}$ such that $s_i\ge 2l-1$ and $s_j\ge l-2$.
Choose subpaths $Q\subseteq P_{s_i}$ and $R\subseteq P_{s_j}$ with
$Q\cong P_{2l-1}$ and $R\cong P_{l-2}$, and let $w$ be the middle
vertex of $Q$. Then $G[\{u\}\cup V(Q)\cup V(R)]$ contains a copy of
$P_{3\cdot l}$ whose unique vertex of degree three is $w$. Therefore,
if $G$ is $P_{3\cdot l}$-free, then both \eqref{eq:capacity} and
\eqref{eq:long} hold.

Conversely, suppose that \eqref{eq:capacity} and \eqref{eq:long} hold, and
assume for a contradiction that $G$ contains a copy $T$ of $P_{3\cdot l}$. Let $w$ be the unique vertex of degree
three in $T$. If $w=u$, then $T-u$ is a linear forest consisting of three paths
of order $l-1$ in $H$. Hence
$\sum_{i=1}^{k}\lfloor s_i/(l-1)\rfloor\ge3$, contradicting
\eqref{eq:capacity}.

Now suppose that $w\ne u$ and $w\in V(P_{s_i})$. Since
$d_T(w)=d_G(w)=3$, the vertex $w$ is internal in $P_{s_i}$.
The edges of $T$ incident with $w$ are $wu$ and the two edges of
$P_{s_i}$ incident with $w$.
Two components of $T-w$ lie in $P_{s_i}$. Together with $w$, they
form a subpath of order $2l-1$. Thus $s_i\ge2l-1$. The third component
of $T-w$ contains $u$ and $l-2$ vertices in some component $P_{s_j}$
of $H$.
If $j\ne i$, then $s_j\ge l-2$, contradicting~\eqref{eq:long}. If
$j=i$, then $P_{s_i}$ contains at least
$(2l-1)+(l-2)=3l-3$ vertices of $T$. Hence $s_i\ge3l-3$ and
$\lfloor s_i/(l-1)\rfloor\ge3$, contradicting~\eqref{eq:capacity}.
This contradiction proves that $G$ is $P_{3\cdot l}$-free.
\end{proof}

For convenience of notation, we define
\[
\begin{aligned}
H_A&=\HOP(3l-4,l-3), && l\ge4,\\
H_B&=\HOPtwo(2l-3,l-2), && l\ge3,\\
H_C&=\HOP(2l-2,l-2), && l\ge3.
\end{aligned}
\]

\begin{lemma}\label{lem:reduction}
If $l\ge3$ and
$
n\ge \max\left\{
1.27\times10^7,\,
6.5025\times 2^{3l-2},\,
(\sqrt{62+3l}+8)^2
\right\},
$
then every extremal graph for $\spex(n,P_{3\cdot l})$ is isomorphic to
$K_1\vee H_B$ or $K_1\vee H_C$, or to $K_1\vee H_A$ when $l\ge4$.
\end{lemma}

\begin{proof}
Let $G$ be an extremal graph for $\spex(n,P_{3\cdot l})$. By
Corollary~\ref{cor:join},
$G=K_1\vee H$, where $H=\bigcup_iP_{s_i}$ is a linear forest of order
$n-1$. Define
$
c(H)=\sum_i\left\lfloor\frac{s_i}{l-1}\right\rfloor.
$
Since $G$ is $P_{3\cdot l}$-free, Lemma~\ref{lem:criterion} gives
$c(H)\le2$.

We first show that $c(H)=2$. Suppose that $c(H)<2$, and let $P_a$ be a
longest component of $H$. If $c(H)=0$, then $a\le l-2$. If $c(H)=1$,
then $a\le2l-3$, and every component other than $P_a$ has order at
most $l-2$. Since $n-1>4(l-1)$, the forest $H$ has another nonempty
component, say $P_b$, where $a\ge b\ge1$. Replacing
$P_a\cup P_b$ with $P_{a+1}\cup P_{b-1}$ (with $P_0$ omitted when $b=1$)
gives a linear forest $H'$ with $c(H')\le2$ and every component of order at
most $2l-2$.
Lemma~\ref{lem:criterion} shows that $K_1\vee H'$ remains
$P_{3\cdot l}$-free, whereas Lemma~\ref{lem:transfer} gives
$\rho(K_1\vee H')>\rho(G)$, a contradiction. Therefore, $c(H)=2$.

Suppose that some component of $H$ has order at least $2l-1$.
By Lemma~\ref{lem:criterion}, its order is at most $3l-4$, and every
other component has order at most $l-3$. Thus $l\ge4$.
Repeated applications of Lemma~\ref{lem:transfer} transform $H$ into
$
 P_{3l-4}\cup qP_{l-3}\cup P_r,
 \quad q\ge0,\quad 0\le r<l-3.
$
By Lemma~\ref{lem:criterion}, every intermediate join is
$P_{3\cdot l}$-free. If $H$ does not already have this form,
Lemma~\ref{lem:transfer} gives a graph with spectral radius greater
than $\rho(G)$, a contradiction. Hence
$
 H=H_A=\HOP(3l-4,l-3),
$

Suppose next that no component has order greater than $2l-2$ and
that one component contributes two to $c(H)$. Its order is exactly
$2l-2$. Repeated applications of Lemma~\ref{lem:transfer} transform
$H$ into
$
 P_{2l-2}\cup qP_{l-2}\cup P_r,
  \quad q\ge0,\quad 0\le r<l-2.
$
By Lemma~\ref{lem:criterion}, every intermediate join remains
$P_{3\cdot l}$-free. If $H$ does not already have this form,
Lemma~\ref{lem:transfer} produces a graph with spectral radius greater
than $\rho(G)$, a contradiction. Hence
$
H=H_C=\HOP(2l-2,l-2).
$

Finally, suppose that two components each contribute one to $c(H)$.
Their orders lie between $l-1$ and $2l-3$. Every other component has
order at most $l-2$. Since $n-1>4l-6$, we may apply
Lemma~\ref{lem:transfer} repeatedly to obtain
$
 2P_{2l-3}\cup qP_{l-2}\cup P_r,
  \quad q\ge0,\quad 0\le r<l-2.
$
By Lemma~\ref{lem:criterion}, every intermediate join remains
$P_{3\cdot l}$-free. If $H$ does not already have this form,
Lemma~\ref{lem:transfer} produces a graph with spectral radius greater
than $\rho(G)$, a contradiction. Hence
$
H=H_B=\HOPtwo(2l-3,l-2).
$

Consequently, $G$ is isomorphic to $K_1\vee H_B$ or
$K_1\vee H_C$, or to $K_1\vee H_A$ when $l\ge4$.
\end{proof}

\begin{lemma}\label{lem:gain}
Suppose that $n\ge1.27\times10^7$, and let $H$ and $H'$ be linear forests
satisfying $V(H)=V(H')$ and $|V(H)|=n-1$. Define
$d=|E(H)\setminus E(H')|$ and $\delta=e(H')-e(H)$. If
$\rho=\rho(K_1\vee H)>36$ and $\delta\rho>4.2d$, then
$\rho(K_1\vee H')>\rho(K_1\vee H)$.
\end{lemma}
\begin{proof}
Let $u$ be the vertex of degree $n-1$ in $K_1\vee H$. Let $\boldsymbol{x}$ be a
positive eigenvector corresponding to $\rho(K_1\vee H)$ with
$\max_{v\in V(K_1\vee H)}x_v=1$. Since $\rho>36$,
Lemma~\ref{lem:coordinates} gives, for any $v,w\in V(H)$,
\[
\frac{1}{\rho^2}\le x_vx_w
\le\left(\frac{1}{\rho}+\frac{2.04}{\rho^2}\right)^2
<\frac{1}{\rho^2}+\frac{4.2}{\rho^3}.
\]

There are $d+\delta$ edges in $E(H')\setminus E(H)$ and $d$ edges in
$E(H)\setminus E(H')$. Applying \eqref{eq:edgecomparison} to
$K_1\vee H$ and $K_1\vee H'$, we obtain
\begin{align*}
\rho(K_1\vee H')-\rho(K_1\vee H)
&\ge
\frac{2}{\boldsymbol{x}^{\mathrm T}\boldsymbol{x}}
\left(
\sum_{vw\in E(H')\setminus E(H)}x_vx_w
-
\sum_{vw\in E(H)\setminus E(H')}x_vx_w
\right)\\
&\ge
\frac{2}{\boldsymbol{x}^{\mathrm T}\boldsymbol{x}}
\left[
\frac{d+\delta}{\rho^2}
-d\left(\frac{1}{\rho^2}+\frac{4.2}{\rho^3}\right)
\right]\\
&=
\frac{2(\delta\rho-4.2d)}
{\rho^3\boldsymbol{x}^{\mathrm T}\boldsymbol{x}}
>0.
\end{align*}
This proves the lemma.
\end{proof}

\begin{lemma}\label{lem:candidate-comparison}
If $l\ge3$ and
$
n\ge \max\left\{
1.27\times10^7,\,
6.5025\times 2^{3l-2},\,
(\sqrt{62+3l}+8)^2
\right\},
$
then $\rho(K_1\vee H_B)>\rho(K_1\vee H_C)$. Moreover, if
$l\ge4$, then $\rho(K_1\vee H_B)>\rho(K_1\vee H_A)$.
\end{lemma}

\begin{proof}
We first compare $H_B$ and $H_C$.
Write $n-1=4l-6+q(l-2)+r$, where $0\le r<l-2$. Then
\[
\begin{aligned}
 H_B&=\HOPtwo(2l-3,l-2)
 =2P_{2l-3}\cup qP_{l-2}\cup P_r,\\
 H_C&=\HOP(2l-2,l-2)
 =P_{2l-2}\cup(q+2)P_{l-2}\cup P_r.
\end{aligned}
\]
In $H_C$, choose a component $P_{2l-2}$ and two components $P_{l-2}$.
Let $z$ be an end vertex of $P_{2l-2}$ and let $w$ be its neighbor.
Let $a$ and $b$ be end vertices of the two copies of $P_{l-2}$,
respectively. Then $H_B$ is obtained from $H_C$ by deleting the edge
$zw$ and adding the edges $az$ and $bz$.

Let $\boldsymbol{x}$ be a positive eigenvector corresponding to
$\rho(K_1\vee H_C)$ with
$\max_{v\in V(K_1\vee H_C)}x_v=1$, and let
$\rho=\rho(K_1\vee H_C)$. Since
$K_1\vee H_C$ contains $K_{1,n-1}$, we have
$\rho\ge\sqrt{n-1}>36$. By \eqref{eq:edgecomparison} and
Lemma~\ref{lem:coordinates},
\begin{align*}
\rho(K_1\vee H_B)-\rho(K_1\vee H_C)
&\ge
\frac{2x_z(x_a+x_b-x_w)}
{\boldsymbol{x}^{\mathrm T}\boldsymbol{x}}\\
&\ge
\frac{2x_z}{\boldsymbol{x}^{\mathrm T}\boldsymbol{x}}
\left[
\frac{2}{\rho}
-\left(\frac{1}{\rho}+\frac{2.04}{\rho^2}\right)
\right]\\
&=
\frac{2x_z(\rho-2.04)}
{\rho^2\boldsymbol{x}^{\mathrm T}\boldsymbol{x}}
>0.
\end{align*}
Therefore, $\rho(K_1\vee H_B)>\rho(K_1\vee H_C)$.

We now compare $H_B$ and $H_A$. Assume $l\ge4$. Let $k_A$ and $k_B$
denote the numbers of components of $H_A$ and $H_B$, respectively. Then
\[
 k_A=1+\left\lceil\frac{n-3l+3}{l-3}\right\rceil,
 \qquad
 k_B=2+\left\lceil\frac{n-4l+5}{l-2}\right\rceil.
\]
A forest of order $n-1$ with $k$ components has $n-1-k$ edges.
Using $z\le\lceil z\rceil\le z+1$ and setting $c=(l-3)(l-2)$, we obtain
\begin{align*}
\delta=e(H_B)-e(H_A)=k_A-k_B
&\ge
-2+\frac{n-3l+3}{l-3}
-\frac{n-4l+5}{l-2}\\
&=
\frac{n-(l-1)^2-2}{c}\\
&\ge
\frac{3(n-1)}{4c}.
\end{align*}
 The last inequality follows because the lower bound on $n$ in
Theorem~\ref{thm:main} implies
$n-1\geq4\bigl((l-1)^2+1\bigr).$

Since $|V(H_A)|=|V(H_B)|=n-1$, we may assume that $V(H_A)=V(H_B)$. Let
$\rho=\rho(K_1\vee H_A)$ and $d=|E(H_A)\setminus E(H_B)|.$
Then $d\le e(H_A)\le n-1$. Since $K_1\vee H_A$ contains
$K_{1,n-1}$, we have $\rho\ge\sqrt{n-1}$.
The lower bound on $n$ in
Theorem~\ref{thm:main} gives
$n-1\ge\bigl(8(l-1)^2+4\bigr)^2$. Hence
$\rho\ge8(l-1)^2+4$. Using the above bound on $\delta$,
\[
 \delta\rho-4.2d
 \ge (n-1)\left(\frac{3\rho}{4c}-4.2\right)>0,
\]
where the strict inequality follows from
\[
 3\rho-16.8c
 \ge3(8(l-1)^2+4)-16.8(l-3)(l-2)
 =7.2l^2+36l-64.8>0.
\]
Lemma~\ref{lem:gain} now gives
$
\rho(K_1\vee H_B)>\rho(K_1\vee H_A).
$
\end{proof}

\begin{proof}[Proof of Theorem~\ref{thm:main}]
By Definition~\ref{def:HOPtwo},
$
H_B=\HOPtwo(2l-3,l-2).
$
Moreover,
$
\sum_{P_s\in H_B}
\left\lfloor\frac{s}{l-1}\right\rfloor=2,
$
and every component of $H_B$ has order at most $2l-3$. Thus,
Lemma~\ref{lem:criterion} implies that $K_1\vee H_B$ is
$P_{3\cdot l}$-free.

Let $G$ be an extremal graph for $\spex(n,P_{3\cdot l})$.
By Lemma~\ref{lem:reduction}, $G$ is isomorphic to
$K_1\vee H_B$, $K_1\vee H_C$, or, when $l\geq4$, $K_1\vee H_A$.
Lemma~\ref{lem:candidate-comparison} shows that
$K_1\vee H_B$ has strictly larger spectral radius than each of the
other possible candidates. Therefore,
$
K_1\vee\HOPtwo(2l-3,l-2)
$
is the unique extremal graph for $\spex(n,P_{3\cdot l})$.
\end{proof}

\section*{Data availability}

No data was used for the research described in the article.

\section*{Declaration of competing interest}

The authors declare that they have no conflict of interest.


\small
\begin{thebibliography}{99}

\bibitem{BollobasNikiforov}
B. Bollob\'{a}s, V. Nikiforov,
Cliques and the spectral radius,
J. Combin. Theory Ser. B 97 (2007) 859--865.

\bibitem{BondyMurty}
J. Bondy, U. Murty,
Graph Theory, Graduate Texts in Mathematics, vol. 244,
Springer, London, 2008.

\bibitem{BootsRoyle}
B. Boots, G. Royle,
A conjecture on the maximum value of the principal eigenvalue of a planar graph,
Geogr. Anal. 23 (1991) 276--282.


\bibitem{BrualdiSolheid}
R. Brualdi, E. Solheid,
On the spectral radius of complementary acyclic matrices of zeros and ones,
SIAM J. Algebraic Discrete Methods 7 (2) (1986) 265--272.

\bibitem{CaoVince}
D. Cao, A. Vince,
The spectral radius of a planar graph,
Linear Algebra Appl. 187 (1993) 251--257.

\bibitem{CvetkovicRowlinson}
D. Cvetkovi\'{c}, P. Rowlinson,
The largest eigenvalue of a graph: a survey,
Linear Multilinear Algebra 28 (1--2) (1990) 3--33.

\bibitem{ChenLiuZhangLinear}
M. Chen, A. Liu, X. Zhang,
Spectral extremal results with forbidding linear forests,
Graphs Combin. 35 (2019) 335--351.

\bibitem{ChenLiuZhangStar}
M. Chen, A. Liu, X. Zhang,
On the spectral radius of graphs without a star forest,
Discrete Math. 344 (4) (2021) 112269.

\bibitem{ChenZhangSurvey}
M. Chen, X. Zhang,
Some new results and problems in spectral extremal graph theory (in Chinese),
J. Anhui Univ. Nat. Sci. 42 (1) (2018) 12--25.

\bibitem{CioabaFengTaitZhang}
S. Cioab\u{a}, L. Feng, M. Tait, X. Zhang,
The maximum spectral radius of graphs without friendship subgraphs,
Electron. J. Combin. 27 (4) (2020) P4--22.



\bibitem{FangLinShi}
L. Fang, H. Lin, Y. Shi,
Extremal spectral results of planar graphs without vertex-disjoint cycles,
J. Graph Theory 106 (3) (2024) 496--524.

\bibitem{FengYuZhang}
L. Feng, G. Yu, X. Zhang,
Spectral radius of graphs with given matching number,
Linear Algebra Appl. 422 (2007) 133--138.


\bibitem{Felsner}
S. Felsner,
Geometric Graphs and Arrangements: Some Chapters from Combinatorial Geometry,
Vieweg, Wiesbaden, 2004.

\bibitem{LinNing}
H. Lin, B. Ning,
A complete solution to the Cvetkovi\'{c}--Rowlinson conjecture,
J. Graph Theory 97 (2021) 441--450.

\bibitem{LiFeng}
Q. Li, K. Feng, On the largest eigenvalue of a graph, Acta Math. Appl. Sin. 2 (2) (1979) 167--175.

\bibitem{LiLiuFengSurvey}
Y. Li, W. Liu, L. Feng,
A survey on spectral conditions for some extremal graph problems,
Adv. Math. (China) 51 (2) (2022) 193--258.

\bibitem{NikiforovPaths}
V. Nikiforov,
The spectral radius of graphs without paths and cycles of specified length,
Linear Algebra Appl. 432 (9) (2010) 2243--2256.

\bibitem{NikiforovSurvey}
V. Nikiforov, Some new results in extremal graph theory, Surveys in Combinatorics, Lond. Math. Soc. Lect. Note Ser. 392 (2011) 141--181.


\bibitem{SchwenkWilson}
A. Schwenk, R. Wilson,
On the eigenvalues of a graph,
in: L. Beineke, R. Wilson (Eds.),
Selected Topics in Graph Theory,
Academic Press, London, 1978, pp. 307--336.

\bibitem{TaitTobin}
M. Tait, J. Tobin,
Three conjectures in extremal spectral graph theory,
J. Combin. Theory Ser. B 126 (2017) 137--161.

\bibitem{WangHuangLin}
X. Wang, X. Huang, H. Lin,
On the spectral extremal problem of planar graphs,
Comput. Appl. Math. 44 (8) (2025) 411.

\bibitem{YinLi}
X. Yin, D. Li,
Spectral extremal problems on outerplanar and planar graphs,
Discrete Math. 349 (1) (2026) 114673.

\bibitem{YinLiMeng}
X. Yin, D. Li, J. Meng,
Planar and outerplanar spectral extremal problems based on paths,
arXiv preprint arXiv:2504.04364.

\bibitem{ZhaiLin}
M. Zhai, H. Lin,
Spectral extrema of graphs: forbidden hexagon,
Discrete Math. 343 (10) (2020) 112028.

\bibitem{ZhaiLiu}
M. Zhai, M. Liu,
Extremal problems on planar graphs without $k$ edge-disjoint cycles,
Adv. Appl. Math. 157 (2024) 102701.

\bibitem{ZhaoHuangLin}
Y. Zhao, X. Huang, H. Lin,
The maximum spectral radius of wheel-free graphs,
Discrete Math. 344 (5) (2021) 112341.

\end{thebibliography}
\end{document}